\documentclass[11pt, a4paper, reqno]{amsart}
\usepackage{amsmath}
\usepackage{amsfonts}
\usepackage{amssymb}
\usepackage{amsthm}
\usepackage{url}
\usepackage{color}
\usepackage{array}
\usepackage{comment}

\usepackage{empheq}

\newcommand{\R}{\mathbb{R}}

\def\vint_#1{\mathchoice%
          {\mathop{\kern 0.2em\vrule width 0.6em height 0.69678ex depth -0.58065ex
                  \kern -0.8em \intop}\nolimits_{\kern -0.4em#1}}%
          {\mathop{\kern 0.1em\vrule width 0.5em height 0.69678ex depth -0.60387ex
                  \kern -0.6em \intop}\nolimits_{#1}}%
          {\mathop{\kern 0.1em\vrule width 0.5em height 0.69678ex depth -0.60387ex
                  \kern -0.6em \intop}\nolimits_{#1}}%
          {\mathop{\kern 0.1em\vrule width 0.5em height 0.69678ex depth -0.60387ex
                  \kern -0.6em \intop}\nolimits_{#1}}}

\newcommand{\loc}{\mathrm{loc}}

\newcommand{\eps}{\varepsilon}

\newcommand{\ch}{\text{\raise 1.3pt \hbox{$\chi$}\kern-0.2pt}}

\theoremstyle{plain}
\newtheorem{theorem}[equation]{Theorem}
\newtheorem{lemma}[equation]{Lemma}
\newtheorem{corollary}[equation]{Corollary}

\numberwithin{equation}{section}

\theoremstyle{definition}
\newtheorem{definition}[equation]{Definition}

\theoremstyle{remark}

\def\Xint#1{\mathchoice%
	{\XXint\displaystyle\textstyle{#1}}%
	{\XXint\textstyle\scriptstyle{#1}}%
	{\XXint\scriptstyle\scriptscriptstyle{#1}}%
	{\XXint\scriptscriptstyle\scriptscriptstyle{#1}}%
	\!\int}
\def\XXint#1#2#3{{\setbox0=\hbox{$#1{#2#3}{\int}$}
		\vcenter{\hbox{$#2#3$}}\kern-.5\wd0}}
\def\intbar{\Xint-}

\newcommand{\bp}[1]{\left( #1 \right)}
\newcommand{\bv}[1]{\left| #1 \right|}

\begin{document}

\title[$p$-Poisson equation: existence and stability]
{Existence and Stability of Solutions of the Dirichlet Problem for the $p$-Poisson Equation in Metric Measure Spaces }

\author{
Luis 
Castillo 
}
\address{
University of Cincinnati, 
Department of Mathematical Sciences,
4199 French Hall West
2815 Commons Way
Cincinnati, USA 
}
\email{castilll@mail.uc.edu}

\author{Timo Takala}
\address{Aalto University, Department of Mathematics and Systems Analysis, P.O. Box 11100 FI-00076 Aalto Finland}
\email{timo.i.takala@aalto.fi}

\thanks{Part of the research was done while L.C. visited Aalto University. 
The visit was supported by the Research Council of Finland project 360184.}

\thanks{T.T. was funded by the Finnish Cultural Foundation.}

\begin{abstract}
We study the Dirichlet problem for the $p$-Poisson equation in the metric measure space setting equipped with a doubling measure and supporting a $(p,p)$-Poincaré inequality.
We prove the existence of the solutions by using a variational approach.
We prove the stability and uniqueness of the solutions, when the space is also geodesic.
\end{abstract}

\maketitle


\medskip


\section{Introduction}

In the Euclidean setting, for a function $u\in C^2 ( \Omega ) $ defined in an open set $\Omega\subset \R^n$, the $p$-Laplacian is defined as
\begin{equation*}\label{p-Laplacian-1}
\Delta_p u :=
\text{div} ( | \nabla  u |^{p-2}  \nabla  u ).
\end{equation*}
Given a function $f \colon \Omega \to \R $ the equation
\begin{equation*}
-\Delta_p u
= 
f \text{ in } \Omega
\end{equation*}
is called the $p$-Poisson equation with source term $f$.
The problem of our interest is the Dirichlet problem for the $p$-Poisson equation
\begin{subequations}\label{p-Dirichlet problem for the p-Poisson equation Euclidean case}
	\begin{empheq}[left=\empheqlbrace]{align}
		- \Delta_p u &= f \text{ in } \Omega,
		\label{p-Dirichlet problem for the p-Poisson equation Euclidean case-1}
		\\
		u &= g \text{ on } \partial \Omega,
		\label{p-Dirichlet problem for the p-Poisson equation Euclidean case-2}
	\end{empheq}
\end{subequations}
where $g : \partial \Omega \to \mathbb{R}$ describes the boundary conditions.
The case $f=0$, which corresponds to the $p$-Laplace equation, is relevant since it is the Euler equation of the integral $\int_{\Omega} \bv{ \nabla u }^p$ and such an integral is the simplest variational integral of non-quadratic growth. 
So \eqref{p-Dirichlet problem for the p-Poisson equation Euclidean case-1} is the natural generalization.

A more general approach to partial differential equations is when the source term is a Radon measure
\begin{equation}\label{p-Laplace equation with measure in the source term}
-\Delta_p u = \mu.
\end{equation}
This approach is detailed more explicitly in \cite{Kolokoltsov2019}.
The equation \eqref{p-Dirichlet problem for the p-Poisson equation Euclidean case-1} is a particular case of \eqref{p-Laplace equation with measure in the source term}, considering the measure induced by $f$.
However, not all partial differential equations of the type \eqref{p-Laplace equation with measure in the source term} can be stated as \eqref{p-Dirichlet problem for the p-Poisson equation Euclidean case-1}.
The general issue in \eqref{p-Laplace equation with measure in the source term} is to represent the measure  with a source function.
In \cite{Mikkonen1996} the Wolff potential is used to get pointwise approximations of a representation for the measure and it is shown that such approximations are appropriate to solve \eqref{p-Laplace equation with measure in the source term}. 
The equation \eqref{p-Laplace equation with measure in the source term} is also studied in \cite{Bungert2023,Heinonen1993}, but in this work we focus on the specific case \eqref{p-Dirichlet problem for the p-Poisson equation Euclidean case-1}.

In \cite{Cianchi2018} the $p$-Poisson equation is studied, and existence and regularity results are provided, when $\Omega$ is convex. 
Solutions to the $p$-Laplace equation are called $p$-harmonic.
The study of such functions is done abstractly considering $\mathcal{A}$-harmonic functions in \cite{Heinonen1993}. This approach is also discussed and extended in \cite{Mikkonen1996}.

In our work we study the Dirichlet problem for the $p$-Poisson equation and the stability of solutions.
The $p$-Poisson equation is studied in \cite{Bungert2023}, but with Neumann conditions, and they make an analysis, where the parameter $p$ tends to infinity, which is an aspect that is not discussed in our work.
The stability is also discussed in \cite{Lindgren2015} in the Euclidean setting considering Hölder regularity.
Considering a suitable positive-valued weight function $w$ one can define the weighted $p$-Laplacian 
\begin{equation*}\label{wigthed p-Laplacian-1}
    \Delta_{w,p} u
    =\text{div} ( w | \nabla  u |^{p-2}  \nabla  u ).
\end{equation*}
This weighted $p$-Laplacian is discussed in \cite{Garain2020}.
It is worth to mention that in \cite{Garain2020} the right-hand side of \eqref{p-Dirichlet problem for the p-Poisson equation Euclidean case-1} depends on the solution $u$.

The weak formulation of \eqref{p-Dirichlet problem for the p-Poisson equation Euclidean case-1} in the Euclidean setting is 
\begin{equation}
\label{p-Poisson in the Euclidean setting weak formulation}
    \int_{ \Omega  }|\nabla u|^{p-2 } \nabla u \cdot \nabla \varphi
=
    \int_{ \Omega } f \varphi
\qquad  
\forall \, \varphi \in W_0^{1,p} (  \Omega  ) . 
\end{equation}
This weak formulation makes sense even if $u$ is not differentiable as long as $u \in W^{1,p}(\Omega)$ and $f \in L^{p'}(\Omega)$, where $p'$ is the Hölder conjugate of $p$.
For the weak formulation of \eqref{p-Dirichlet problem for the p-Poisson equation Euclidean case-2}, we require $g\in W^{1,p} ( \Omega ) $ and $ u-g \in W^{1,p} ( \Omega ) $ such that on the boundary the value of this function is zero.
Thus the weak formulation of \eqref{p-Dirichlet problem for the p-Poisson equation Euclidean case-2} is
\begin{equation}
\label{Reasonably extension of the boundary condition}
    u-g \in W^{1,p}_0  ( \Omega ). 
\end{equation}

We aim to formulate the Dirichlet problem for the $p$-Poisson equation in the metric measure space setting. 
From the weak formulations \eqref{p-Poisson in the Euclidean setting weak formulation} and \eqref{Reasonably extension of the boundary condition}, it follows that to state the Dirichlet problem, we would need to find concepts to replace $\nabla  u$, $| \nabla u|$ and $W^{1,p}_0  ( \Omega ) $ in the metric measure space setting.
The upper gradients are one form of generalizing the magnitude of the derivative in metric measure spaces. However, since $\nabla u$ is a vector, the replacement of this is more delicate. A way to do this is with Cheeger's differential structure \cite{Cheeger1999}.
The space $W^{1,p}_0  ( \Omega ) $ is replaced by the Newtonian space with zero boundary values $N_0^{1,p} (\Omega )$.

In the nonsmooth setting the proofs of potential theory as in  \cite{Heinonen1993,Mikkonen1996} are done considering minimizers of the energy. In \cite{Gong2012,nages} the minimizing property is considered as the definition of $p$-harmonicity in metric measure spaces
and similarly we define solutions to the $p$-Poisson equation in the nonsmooth setting by using minimizers.
In the context of RCD spaces the $p$-Poisson equation is studied in \cite{Benatti2025} and it is shown that the solutions belong to $W^{2,2} (X) $, which is something that is not defined with Newtonian spaces.
The Neumann problem for the $p$-Laplace equation has been studied in the nonsmooth setting in \cite{Capogna2026,Capogna2022}, where they obtain existence, global Hölder regularity and stability of solutions.

In the Cheeger's setting we show that $u$ is a solution in the sense of \eqref{Variational formulation of the p-Laplace equation}, which is a metric measure space version of \eqref{p-Poisson in the Euclidean setting weak formulation}, if and only if $u$ is a minimizer of the functional
\begin{equation*}
J_C(u)
= \int_\Omega |\nabla u|_x^p d \mu(x) - p \int_\Omega f u d \mu.
\end{equation*}
Thus it is possible to define the solution as a minimizer of the functional $J_C$.
Note that $J_C$ includes the function $u$ and the magnitude of the Cheeger's gradient $|\nabla u|_x$, but not the vector $\nabla u$ itself. Thus this approach of minimizing $J_C$ makes sense also if we use upper gradients instead of Cheeger's gradients.

The paper is organized as follows. Section \ref{sect-prelim} is to recall basic properties of metric measure spaces.
In Section \ref{formulationsection} we discuss how to state the problems in the general metric measure space setting by using the functional $J$.
In Section \ref{sect-existence} we discuss the existence of solutions.
In Section \ref{sect-stability} we discuss the stability of the solutions in the Cheeger's setting, where the space is geodesic.
As a corollary we obtain that the solution is unique.

\section{Preliminaries}
\label{sect-prelim}

Throughout this paper we let $(X,d,\mu)$ be a metric measure space with $\mu$ a Borel regular measure.
For $x \in X$ and $r>0$ the open ball is denoted by $B(x,r) := \{ y \in X : d(x,y) < r \}$. For $a > 0$ we denote $a B(x,r) = B(x,ar)$.
For a $\mu$-measurable set $E \subseteq X$, with $0 < \mu(E) < \infty$, and $f \in L^1(E)$ we denote the \emph{integral average} of $f$ over $E$ by
\begin{equation*}
f_E
:= \vint_E f d \mu
:= \frac{1}{\mu(E)} \int_E f d \mu.
\end{equation*}

For a function $u : X \rightarrow \mathbb{R}$ a Borel measurable function $g : X \rightarrow [0,\infty]$ is an \emph{upper gradient of $u$}, if
\begin{equation}
\label{uppergradient}
|u(\gamma(a)) - u(\gamma(b))|
\leq \int_{\gamma} g(\cdot) ds
\end{equation}
for all nonconstant compact rectifiable curves $\gamma : [a,b] \rightarrow X$.
If the inequality \eqref{uppergradient} is satisfied for all curves except for a curve family that has $p$-modulus zero, then $g$ is a \emph{$p$-weak upper gradient of $u$}.
If a function $u$ has a $p$-weak upper gradient $g \in L^p(X)$, then there exists a unique $p$-weak upper gradient $g_u$ of $u$ with the smallest $L^p$ norm among the $L^p$ norms of all $p$-weak upper gradients of $u$. This function $g_u$ is called the \emph{minimal $p$-weak upper gradient of $u$}.
See \cite{Heinonen2015} 
for more on path integrals, the $p$-modulus of curve families and minimal $p$-weak upper gradients.

\begin{definition}
Let $(X,d,\mu)$ be a metric measure space and $p \geq 1$. We say that $u \in N^{1,p}(X)$, if $u \in L^p(X)$ and there exists a $p$-weak upper gradient $g$ of $u$ such that $g \in L^p(X)$. The space $N^{1,p}(X)$ is known as the \emph{Newtonian space} and it is equipped with the norm
\begin{equation*}
\|u\|_{N^{1,p}(X)}
:= \left( \int_X |u|^p + g_u^p d \mu \right)^{1/p}.
\end{equation*}
\end{definition}

\begin{definition}
Let $(X,d,\mu)$ be a metric measure space, $\Omega \subset X$ and $p \geq 1$. The \emph{Newtonian space with zero boundary values} $N_0^{1,p}(\Omega)$ consists of functions $u \in N^{1,p}(X)$, restricted to $\Omega$, such that $u = 0$ in $X \setminus \Omega$.
\end{definition}
See \cite{bjornbjorn,nages} for more information on the space $N_0^{1,p}$.

\begin{definition}
Let $(X,d,\mu)$ be a metric measure space and $p \geq 1$. The \emph{capacity} of a set $E \subset X$ is the number
\begin{equation*}
C_p(E)
:= \inf \|u\|_{N^{1,p}(X)}^p,
\end{equation*}
where the infimum is taken over all $u \in N^{1,p}(X)$ such that $u \geq 1$ on $E$.
\end{definition}

\begin{definition}
Let $(X,d,\mu)$ be a metric measure space such that $\mu(B)$ is positive for every ball $B \subset X$.
The space $(X,d,\mu)$ \emph{supports a $(q,p)$-Poincaré inequality} with $p \geq 1$ and $q \geq 1$, if there exist constants $C_P > 0$ and $\lambda \geq 1$ such that for every $u \in L^1(X)$, $x \in X$, $r > 0$ and $g$ an upper gradient of $u$ we have
\begin{equation*}
\left( \vint_{B(x,r)} |u - u_{B(x,r)}|^q d \mu \right)^{1/q}
\leq C_P r \left( \vint_{B(x,\lambda r)} g^p d \mu \right)^{1/p}.
\end{equation*}
\end{definition}

\begin{definition}
A Borel regular measure $\mu$ on the metric space $(X,d)$ is \emph{doubling}, if there exists a constant $C_{\mu} \geq 1$ such that
\begin{equation*}
0 < \mu(B(x,2r)) \leq C_{\mu} \mu(B(x,r)) < \infty
\end{equation*}
for every $x \in X$ and $r > 0$.
\end{definition}

Now we introduce Maz'ya's inequality, which plays an essential role in the estimations discussed in this paper.

\begin{theorem}
\emph{(Maz'ya's inequality, \cite[Theorem 5.53]{bjornbjorn})}
Assume that the space $(X,d,\mu)$ supports a $(q,p)$-Poincaré inequality for some $q \geq p$. For $u \in N_{loc}^{1,p}(X)$, let $S := \{ x \in X : u(x) = 0 \}$. Then for all balls $B = B(x,r)$,
\begin{equation*}
\left( \vint_{2B} |u|^q d\mu \right)^{p/q}
\leq \frac{C(r^p+1)}{C_p(B \cap S)} \int_{2 \lambda B} g_u^p d\mu,
\end{equation*}
where $\lambda$ is the dilation constant in the $(q,p)$-Poincaré inequality.
The constant $C$ depends only on $p$ and the constant $C_P$ from the $(q,p)$-Poincaré inequality.
\end{theorem}

The next lemma follows from Maz'ya's inequality and it is used many times in this paper.

\begin{lemma}
\label{mazyalemma}
Let $(X,d,\mu)$ be a metric measure space, where the measure $\mu$ is doubling and the space supports a $(p,p)$-Poincaré inequality with $p \geq 1$.
Let $\Omega \subset X$ be a nonempty bounded open connected set so that $\mu(X \setminus \Omega) > 0$.
Let $u \in N_0^{1,p}(\Omega)$. Then there exists a constant $c_1$ that depends only on $p$, $C_P$, $\Omega$ and $\mu$ such that
\begin{equation*}
\|u\|_{L^p(\Omega)}
\leq c_1 \|g_u\|_{L^p(\Omega)}.
\end{equation*}
\end{lemma}

\begin{proof}
We use Maz'ya's inequality to prove this result. By considering the zero extension of $u$, we have $X \setminus \Omega \subset S$. We choose $x \in \Omega$ and the radius $r$ so big that $\Omega \subset 2B$ and $\mu(B \setminus \Omega) > 0$. Then $C_p(B \cap S) \geq C_p(B \setminus \Omega) \geq \mu(B \setminus \Omega)$ and thus
\begin{align*}
\|u\|_{L^p(\Omega)}^p
&= \int_{2B} |u|^p d\mu
\leq \frac{C(r^p+1)\mu(2B)}{C_p(B \cap S)} \int_{2 \lambda B} g_u^p d\mu
\\
&\leq \frac{C(r^p+1)\mu(2B)}{\mu(B \setminus \Omega)} \int_{\Omega} g_u^p d\mu.
\end{align*}
Note that $g_u = 0$ $\mu$-almost everywhere in the complement of $\Omega$, because $u=0$ in the complement of $\Omega$, see \cite[Corollary 2.21]{bjornbjorn}.
Thus by denoting
\begin{equation*}
c_1 = \left( \frac{C(r^p+1)\mu(2B)}{\mu(B \setminus \Omega)} \right)^{1/p},
\end{equation*}
we get the result.
\end{proof}

\subsection*{General hypothesis}

From now on in this paper $(X,d,\mu)$ is a complete metric measure space, where the measure $\mu$ is doubling and the space supports a $(p,p)$-Poincaré inequality with $1 < p < \infty$.
We assume that $\Omega \subset X$ is a nonempty bounded open connected set so that $\mu(X \setminus \Omega) > 0$. 

\section{Formulation of the Dirichlet problem in metric measure spaces}
\label{formulationsection}

Here we will extend the $p$-Poisson equation into the nonsmooth metric measure setting. To do that we need to define the notion of Cheeger's gradient.
The original idea appears in \cite{Cheeger1999}. However, the explicit statement of the result that we use comes from \cite{Heinonen2015}.
We will provide a summary of the construction of the theorem.
The idea is to construct a differentiable structure like a manifold but for metric measure spaces. 
Consider the following definitions.

\begin{definition}
A function $u \in N_{\loc}^{1,p}\bp{X}$ is \emph{asymptotic $p$-harmonic at $x_0 \in~X$}, if
\begin{equation*}
\lim_{r \to 0}
\bp{
\intbar_{B\bp{x_0,r}} g_u^p \; d\mu
-
\inf_{\varphi \in N_c^{1,p}\bp{B\bp{x_0,r}}}
\intbar_{B\bp{x_0,r}} g_{u+\varphi}^p \; d\mu
}
=
0.
\end{equation*}
Here $N_c^{1,p}(B)$ is the space of Newtonian functions that have compact support in $B$.
\end{definition}

\begin{definition}
A function $u\in N_{\loc}^{1,p}(X)$ is \emph{asymptotically generalized linear at $x_0 \in X$},
\index{Asymptotic!generalized linear}
\index{Function!asymptotic generalized linear}
if $u$ is asymptotic $p$-harmonic at $x_0$ and $x_0$ is a Lebesgue point of $g_u^p$. 
\end{definition}

\begin{theorem}
\label{Existence of differentiable structure on mms throu asym gen linear}
\emph{(\cite[Theorem 13.4.4]{Heinonen2015})}
If $X$ is geodesic, there exist a positive integer $N$ and a countable collection of measurable sets $\left\{U_\alpha\right\}_\alpha$ with $\mu\left(U_\alpha\right)>~0$ satisfying $\mu\left(X \backslash \bigcup_\alpha U_\alpha\right)=0$, and, for each $U_\alpha$, there is a collection of 1-Lipschitz functions $f_1^\alpha, \ldots, f_{N(\alpha)}^\alpha$ on $X$, where $N(\alpha) \leq N$ for each $\alpha$, such that
\begin{enumerate}
    \item 
    each $f_j^\alpha, j=1, \ldots, N(\alpha)$, is asymptotically generalized linear at each point of $U_\alpha$,
    
\item for each $\vec{a} \in \mathbb{R}^{N(\alpha)} \backslash\{\overrightarrow{0}\}$ and each $x_0 \in U_\alpha$ we have $g_{\vec{a} \cdot \overrightarrow{f^\alpha}}\left(x_0\right)>0$,

\item whenever $u: X \rightarrow \mathbb{R}$ is an L-Lipschitz function, there is a set $V_\alpha(u) \subset U_\alpha$, with $\mu\left(U_\alpha \backslash V_\alpha(u)\right)=0$, and Borel functions $b_j^\alpha(u): V_\alpha(u) \rightarrow \mathbb{R}$, $j=1, \ldots, N(\alpha)$, such that for $x_0 \in V_\alpha(u)$ we have $g_{\vec{a} \cdot \vec{f}^\alpha-u}\left(x_0\right)=0$ if and only if

$$
\vec{a}=\vec{b}^\alpha(u)\left(x_0\right)=\left(b_1^\alpha(u)\left(x_0\right), \ldots, b_{N(\alpha)}^\alpha(u)\left(x_0\right)\right) .
$$
    
\end{enumerate}
\end{theorem}

The function $\vec b^\alpha(u): V_\alpha(u) \rightarrow \mathbb{R}^{N(\alpha)}$ is called a derivative of $u$ and is denoted $D^\alpha u$. 
We may assume without loss of generality that the sets $U_\alpha$ are pairwise disjoint and thus we define
\begin{equation}
\label{Differential in geodesic PI spaces}
    D(u):=\sum_\alpha D^\alpha u \chi_{U_\alpha} .
\end{equation}
The operator $D$ is a linear operator for Lipschitz functions \cite{Heinonen2015}.

\begin{theorem}
\label{Differentiable structure on mms throu asym gen linear properties}
\emph{(\cite[Proposition 13.5.8]{Heinonen2015})}
Suppose that $X$ is geodesic.
Then there is a positive integer $N$ and a bounded linear differential operator
$$
\nabla: N^{1, p}(X) \rightarrow L^p\left(X: \mathbb{R}^N\right)
$$
such that $\nabla(u v)=u \nabla v+v \nabla u$ whenever $u, v \in N^{1, p}(X)$ are Lipschitz functions, 
\begin{equation*}
\left|\nabla u\left(x_0\right)\right|_{x_0} \leq g_u\left(x_0\right) \leq \sqrt{N}\left|\nabla u\left(x_0\right)\right|_{x_0}
\end{equation*}
for almost every $x_0 \in X$, and $\nabla u$ coincides, for Lipschitz functions $u$, with the operator defined in \eqref{Differential in geodesic PI spaces}.
\end{theorem}

The existence of an inner product $\langle \cdot,\cdot \rangle_{x_0}$ on $\mathbb{R}^N$, that induces the norm $|\cdot|_{x_0}$, is proved within the proof of \cite[Theorem 13.5.7]{Heinonen2015}.
These theorems provide the existence of the differentiable structure.
However the differentiable structure of Theorem \ref{Existence of differentiable structure on mms throu asym gen linear} is not necessarily unique.
Theorem \ref{Differentiable structure on mms throu asym gen linear properties} shows us that if $u \in N^{1,p}(X)$, then $|\nabla u|_x \in L^p(X)$, where $\nabla u$ is the Cheeger's gradient, which is linear.
In particular
\begin{equation*}
\|g_u\|_{L^p(X)}
\leq \sqrt{N} \|\nabla u\|_{L^p(X)}
:= \sqrt{N} \left( \int_X (|\nabla u(x)|_x)^p d \mu(x) \right)^{1/p}.
\end{equation*}
Therefore, if $u \in N_0^{1,p}(\Omega)$ and $X$ is geodesic, we get from Lemma \ref{mazyalemma} that
\begin{equation}
\label{mazyalemmacheeger}
\|u\|_{L^p(\Omega)}
\leq c_1 \sqrt{N} \|\nabla u\|_{L^p(\Omega)}.
\end{equation}

Now we can define the solution to the Dirichlet problem for the $p$-Poisson equation in the Cheeger's setting.
Here we also replace the Sobolev space in \eqref{Reasonably extension of the boundary condition} with the Newtonian space.

\begin{definition}
\label{solutiondefinitioncheeger}
Suppose that $X$ is geodesic and fix a differentiable structure of Theorem \ref{Existence of differentiable structure on mms throu asym gen linear}. 
A solution (in the Cheeger's setting) for the $p$-Poisson equation with source term $f \in L^{p'}(\Omega)$ is a function $u \in N^{1,p} (\Omega)$ such that
\begin{equation}
\label{Variational formulation of the p-Laplace equation}
    \int_{ \Omega  }|\nabla u|_x^{p-2} \langle \nabla u , \nabla \varphi \rangle_x d \mu(x)
=
    \int_{ \Omega } f \varphi d \mu
\qquad  
\forall \, \varphi \in N_0^{1 , p } ( \Omega ) . 
\end{equation}
Here $\nabla u$ and $\nabla \varphi$ denote Cheeger's gradients.
A solution to the Dirichlet problem for the $p$-Poisson equation (in the Cheeger's setting) with source term $f$ and boundary data $g \in N^{1,p}(\Omega)$ is a function $u \in N^{1,p}(\Omega)$ such that $u-g \in N_0^{1,p}(\Omega)$ and \eqref{Variational formulation of the p-Laplace equation} holds.
\end{definition}

In the Cheeger's setting $u$ being a solution in the sense of the previous definition is actually equivalent to $u$ being a minimizer of the functional
\begin{equation*}
J_C(u)
:= \int_\Omega |\nabla u|_x^p d \mu(x)  - p \int_\Omega f u d \mu.
\end{equation*}
This is proved below by finding the critical points of $J_C$.
More precisely we fix a direction $\varphi \in N_0^{1,p} (\Omega )$ and define the auxiliary function $L : \mathbb{R} \rightarrow \mathbb{R}$ by
\begin{equation}
\label{Definition of L}
L( \varepsilon  )
:= J_C (u + \varepsilon\varphi ) 
= \int_\Omega |\nabla ( u  + \varepsilon \varphi )|_x^{p} - p f ( u + \varepsilon \varphi )d \mu(x).
\end{equation}
We find critical points of $J_C$ by checking the critical points of $L$.
In the next lemma we prove the differentiability of $L$.
\begin{lemma}
\label{Derivative of the auxiliary functional L-lemma}
Let $X$ be geodesic and fix a differentiable structure of Theorem \ref{Existence of differentiable structure on mms throu asym gen linear}.
Let $u \in N^{1,p}(\Omega)$, $f \in L^{p'}(\Omega)$ and $\varphi \in N_0^{1,p} (\Omega )$.
Then the function $L$ is differentiable and furthermore
\begin{equation}
\label{Derivative of the auxiliary functional L}
    L'( \varepsilon )
    =
    p
    \int_\Omega 
        |\nabla ( u  + \varepsilon \varphi )|_x^{p-2} 
        \langle \nabla (u + \varepsilon \varphi ) , \nabla \varphi \rangle_x
    d \mu(x)
    - p \int_\Omega f\varphi d \mu.
\end{equation}
\end{lemma}
\begin{proof}
The claim follows from \cite[Theorem 2.27]{Folland1999}, which allows us to move the derivative inside the integral.
To use this theorem we need to show that the partial derivative of the integrand in \eqref{Definition of L} with respect to $\varepsilon$ is bounded by an integrable function in $\Omega$.
For the second term
\begin{equation*}
     \frac{\partial}{ \partial \varepsilon }  p f ( u + \varepsilon \varphi )
     =
     p f \varphi,
\end{equation*}
which is an integrable function by Hölder's inequality. 
Now we prove the same for the first term.
Notice that
\begin{equation*}
    \frac{\partial }{ \partial \varepsilon }
    |\nabla ( u  + \varepsilon \varphi )|_x^{p}  
=
    p
    |\nabla ( u  + \varepsilon \varphi) |_x^{p-2}  
    \langle \nabla (u + \varepsilon \varphi ) , \nabla \varphi \rangle_x.
\end{equation*}
From the Cauchy-Schwarz and Young's inequalities we get
\begin{equation*}
    \left|
    \frac{\partial }{ \partial \varepsilon }
    |\nabla ( u  + \varepsilon \varphi )|_x^{p}  
    \right|
\leq
    p
    |\nabla ( u  + \varepsilon \varphi) |_x^{p-1}  
        | \nabla \varphi |_x
\leq
    | \nabla \varphi |_x^p
    +
    (p-1)|\nabla ( u  + \varepsilon \varphi) |_x^{p}.
\end{equation*}
Then from the convexity of the function $t \to t^p$, we have
\begin{equation*}
    \left|
    \frac{\partial }{ \partial \varepsilon }
    |\nabla ( u  + \varepsilon \varphi )|_x^{p}  
    \right|
\leq
    | \nabla \varphi |_x^p
    +
    (p-1) 2^{p-1} 
    \left( 
    | \nabla u |_x^{p}
    +
    |\varepsilon|^p | \nabla \varphi |_x^{p} 
    \right)
    .
    \end{equation*}
Therefore $\frac{\partial }{ \partial \varepsilon }|\nabla ( u  + \varepsilon \varphi )|_x^{p}$ is bounded by an integrable function in $\Omega$. Thus we conclude that $L$ is differentiable.
Finally \eqref{Derivative of the auxiliary functional L} follows from the calculations above.
\end{proof}

To find the critical points of $J_C$ we need the following lemma.
\begin{lemma}
\label{notclarkson}
Let $V$ be an inner product space over the field $\mathbb{R}$ with the induced norm $\|\cdot\| := \sqrt{\langle \cdot , \cdot \rangle}$.
Then for all $a,b \in V$ the following estimations hold:
    \begin{equation*}
	\left\langle \|b\|^{p-2} b-\|a\|^{p-2} a , b-a \right\rangle
	\geq 
	\begin{cases}
		2^{1-p} \|b-a\|^p, & p \geq 2, \\ 
		(p-1) (\|a\|+\|b\|)^{p-2}\|b-a\|^2, & p < 2.
	\end{cases}
\end{equation*}
\end{lemma}
\begin{proof}
For the case $1 < p < 2$ the following inequality is proved in \cite[Chapter 12]{Lindqvist2017}:
\begin{equation*}
    \left\langle \|b\|^{p-2} b - \|a\|^{p-2} a , b-a \right\rangle
    \geq 
    (p-1) \|b-a\|^2 \int_0^1 \|a + t(b-a)\|^{p-2} dt.
\end{equation*}
From the above inequality and the triangle inequality
\begin{equation*}
\|a + t(b-a)\|
\leq (1-t)\|a\| + t\|b\|
\leq \|a\| + \|b\|
\end{equation*}
together with the fact that $p-2 < 0$, we conclude that the claim holds.

The case $p \geq 2$ is also explained in \cite[Chapter 12]{Lindqvist2017}. There they obtain the larger coefficient $2^{2-p}$ instead of $2^{1-p}$. However there is a mistake in the proof. Their proof works, if we have the coefficient $2^{1-p}$ instead of $2^{2-p}$, which is why we have that here.
Technically in \cite{Lindqvist2017} the results are proven for vectors in $\mathbb{R}^n$ with the dot product, but the same proofs work also for a general inner product space.
\end{proof}

Now we prove the minimizing property of solutions.
\begin{theorem}
\label{Minimizing property of solutions}
Let $X$ be geodesic and fix a differentiable structure of Theorem \ref{Existence of differentiable structure on mms throu asym gen linear}.
Let $f \in L^{p'}(\Omega)$ and $u , g \in N^{1,p}(\Omega)$ such that $u-g \in N_0^{1,p}(\Omega)$. Then $u$ satisfies \eqref{Variational formulation of the p-Laplace equation} if and only if $J_C(u) \leq J_C(v)$ for every $v \in N^{1,p}(\Omega)$ such that $v-g \in N_0^{1,p}(\Omega)$.
\end{theorem}
\begin{proof}
First assume that $u$ satisfies \eqref{Variational formulation of the p-Laplace equation} and let $v \in N^{1,p}(\Omega)$ such that $v-g \in N_0^{1,p}(\Omega)$.
To prove that $J_C(u) \leq J_C(v)$ we let $\varphi = v-u$.
Then we have $\varphi = (v-g)-(u-g) \in N_0^{1,p}(\Omega)$ and $L(1) = J_C(v)$.
For every positive $\varepsilon$ we get from \eqref{Derivative of the auxiliary functional L}
\begin{align*}
L'(\varepsilon)
&= \int_{\Omega} p |\nabla(u + \varepsilon \varphi)|_x^{p-2} \langle \nabla(u+\varepsilon \varphi) , \nabla \varphi \rangle_x d \mu(x) - p \int_{\Omega} f \varphi d \mu
\\
&= \int_{\Omega} p |\nabla(u + \varepsilon \varphi)|_x^{p-2} \langle \nabla(u+\varepsilon \varphi) , \nabla \varphi \rangle_x d \mu(x)
\\
&- p \int_{\Omega} |\nabla u|_x^{p-2 } \langle \nabla u , \nabla \varphi \rangle_x d \mu(x)
\\
= \frac{p}{\varepsilon} &\int_{\Omega} \langle |\nabla(u + \varepsilon \varphi)|_x^{p-2} \nabla(u + \varepsilon \varphi) - |\nabla u|_x^{p-2} \nabla u , \nabla(u + \varepsilon \varphi) - \nabla u \rangle_x d \mu(x)
\\
&\geq 0,
\end{align*}
where in the final step we used Lemma \ref{notclarkson}.
In particular this implies that $L(1) \geq L(0)$, which is what we wanted to show.

Now suppose that $J_C(u) \leq J_C(v)$ for every $v \in N^{1,p}(\Omega)$ such that $v-g \in N_0^{1,p}(\Omega)$. Let $\varphi \in N_0^{1,p}(\Omega)$. Then by setting $v = u + \varepsilon \varphi$ we have $v - g \in N_0^{1,p}(\Omega)$ for every $\varepsilon \in \mathbb{R}$.
Thus $L(0) = J_C(u) \leq J_C(v) = L(\varepsilon)$ for every $\varepsilon$, and $L$ attains its minimum at 0. Therefore we must have $L'(0)=0$ and thus
\begin{equation*}
L'(0)
= \int_{\Omega} p |\nabla u|_x^{p-2} \langle \nabla u , \nabla \varphi \rangle_x d \mu(x) - p \int_{\Omega} f \varphi d \mu
= 0,
\end{equation*}
and \eqref{Variational formulation of the p-Laplace equation} follows from this.
\end{proof}

This approach of minimizing the functional $J_C$ has the benefit that $J_C$ only includes the function $u$ and the magnitude of the Cheeger's gradient $|\nabla u|_x$, but not the vector $\nabla u$ itself.
Thus this approach also makes sense even if we don't have a Cheeger differentiable structure, so even if $X$ is not geodesic. In this case we consider the minimal $p$-weak upper gradient instead of the Cheeger's gradient. This motivates the following definition of solutions to the Dirichlet problem.
\begin{definition}
\label{solutionmindef}
Let $f \in L^{p'}(\Omega)$ and define the functional $J : N^{1,p}(\Omega) \to \R$ by
\begin{equation*}
J(u)
:= \int_\Omega g_u^p d \mu - p \int_\Omega f u d \mu.
\end{equation*}
Here $g_u$ is the minimal $p$-weak upper gradient of $u \in N^{1,p}(\Omega)$.
A solution to the Dirichlet problem for the $p$-Poisson equation with source term $f$ and boundary data $g \in N^{1,p}(\Omega)$ is defined as a function $u \in N^{1,p}(\Omega)$ such that $u-g \in N_0^{1,p}(\Omega)$ and $u$ minimizes $J$, i.e. $J(u) \leq J(v)$ for every $v \in N^{1,p}(\Omega)$ such that $v-g \in N_0^{1,p}(\Omega)$.
\end{definition}

\section{Existence of solutions}
\label{sect-existence}

In this section we prove that a solution in the sense of Definition \ref{solutionmindef} exists. In particular we don't need to assume that the space is geodesic.

\begin{theorem}
Let $f \in L^{p'}(\Omega)$ and $g \in N^{1,p}(\Omega)$.
There exists a solution $u_0$ to the Dirichlet problem for the $p$-Poisson equation, that is to say $u_0 \in~N^{1,p}(\Omega)$ such that $J(u_0) = \inf\{ J(u) : u-g \in N_0^{1,p}(\Omega) \}$ and $u_0-g \in N_0^{1,p}(\Omega)$.
\end{theorem}

\begin{proof}
The first step is to show that the functional $J$ is bounded from below. Thus assume that $u-g \in N_0^{1,p}(\Omega)$. By using Hölder's inequality and Minkowski's inequality we get
\begin{equation*}
\left| \int_{\Omega} f u d\mu \right|
\leq \|f\|_{L^{p'}(\Omega)} \|u\|_{L^p(\Omega)}
\leq \|f\|_{L^{p'}(\Omega)} (\|u-g\|_{L^p(\Omega)} + \|g\|_{L^p(\Omega)}).
\end{equation*}
Next we use Lemma \ref{mazyalemma} for the function $u-g$.
A useful property of upper gradients is that $g_u + g_g$ is a $p$-weak upper gradient of $u-g$. Thus we have
\begin{equation*}
\|u-g\|_{L^p(\Omega)}
\leq c_1 \|g_{u-g} \|_{L^p(\Omega)}
\leq c_1 \|g_u+g_g\|_{L^p(\Omega)}
\leq c_1 \|g_u\|_{L^p(\Omega)} + c_1 \|g_g\|_{L^p(\Omega)}.
\end{equation*}
By putting these estimates into the functional $J(u)$, we get
\begin{align*}
J(u)
&= \|g_u\|_{L^p(\Omega)}^p - p \int_{\Omega} f u d\mu
\\
&\geq \|g_u\|_{L^p(\Omega)}^p - p \|f\|_{L^{p'}(\Omega)} (c_1 \|g_u\|_{L^p(\Omega)} + c_1 \|g_g\|_{L^p(\Omega)} + \|g\|_{L^p(\Omega)})
\\
&\geq (1-p) ( c_1 \|f\|_{L^{p'}(\Omega)})^{\frac{p}{p-1}} - p \|f\|_{L^{p'}(\Omega)} (c_1\|g_g\|_{L^p(\Omega)} + \|g\|_{L^p(\Omega)}).
\end{align*}
The last inequality is obtained by finding the minimum of the function $[0,\infty) \ni t \to t^p - p c_1 \|f\|_{L^{p'}(\Omega)} t$ by using derivatives.
Thus we have found a lower bound for $J(u)$ and therefore
$$ I : =  \inf\{J(u) : u-g \in N_0^{1,p}(\Omega)\} > -\infty.$$
Also clearly $I \leq J(g) < \infty$.

Now we will prove that there exists $u_0 \in N^{1,p}(\Omega)$ so that $u_0-g \in N_0^{1,p}(\Omega)$ and $u_0$ reaches the infimum.
Let $\{ u_n \}_{n=1}^{\infty}$ be a sequence in $N^{1,p}(\Omega)$ such that $u_n - g \in N_0^{1,p}(\Omega)$ for every $n$ and $J(u_n)\to I$.

By passing to a subsequence if necessary, we can assume that $I \leq J(u_n) \leq I + 1$ for every $n$. Then
\begin{equation*}
\|g_{u_n}\|_{L^p(\Omega)}^p
\leq I + 1 + p \int_{\Omega} f u_n d\mu.
\end{equation*}
Since $u_n - g \in N_0^{1,p}(\Omega)$, we get the same way as in the proof of the lower bound for $J(u)$ that
\begin{equation*}
\left| \int_{\Omega} f u_n d\mu \right|
\leq \|f\|_{L^{p'}(\Omega)} (\|g\|_{L^p(\Omega)} + c_1 \|g_{u_n}\|_{L^p(\Omega)} + c_1 \|g_g\|_{L^p(\Omega)}).
\end{equation*}
Therefore
\begin{equation*}
\|g_{u_n}\|_{L^p(\Omega)}^p
\leq |I| + 1 + p \|f\|_{L^{p'}(\Omega)} (\|g\|_{L^p(\Omega)} + c_1 \|g_{u_n}\|_{L^p(\Omega)} + c_1 \|g_g\|_{L^p(\Omega)}).
\end{equation*}
If $\|g_{u_n}\|_{L^p(\Omega)} > 1$, we get by dividing both sides with $\|g_{u_n}\|_{L^p(\Omega)}$
\begin{equation*}
\|g_{u_n}\|_{L^p(\Omega)}^{p-1}
\leq |I| + 1 + p \|f\|_{L^{p'}(\Omega)} (\|g\|_{L^p(\Omega)} + c_1 + c_1 \|g_g\|_{L^p(\Omega)}),
\end{equation*}
which is trivially true also if $\|g_{u_n}\|_{L^p(\Omega)} \leq 1$.
For the $L^p$ norm of $u_n$ we get as in the proof of the lower bound for $J(u)$ that
\begin{equation*}
\|u_n\|_{L^p(\Omega)}
\leq \|g\|_{L^p(\Omega)} + c_1 \|g_{u_n}\|_{L^p(\Omega)} + c_1 \|g_g\|_{L^p(\Omega)}.
\end{equation*}
This shows that the sequence $u_n$ is bounded in the $N^{1,p}(\Omega)$ norm.

Since $g_{u_n}$ is bounded in the $L^p$ norm, by passing to a subsequence if necessary, we can assume that $\| g_{u_n} \|_{L^p(\Omega)}$ converges.
Then because $\| g_{u_n} \|_{L^p(\Omega)}$ converges and $J(u_n)$ converges, $\int_{\Omega} f u_n d\mu$ also converges.

Since $p>1$, we have that $L^p(\Omega)$ is reflexive \cite{Heinonen2015}, and so we can apply Mazur's lemma \cite[Theorem 7.3.8]{Heinonen2015}.
Then there exists a sequence of convex combinations of $\{ u_n \}_{n=1}^{\infty}$, denoted $\{ v_n \}_{n=1}^{\infty}$, and a sequence of convex combinations of $g_{u_n}$, denoted $\rho_n$, and functions $u_0 \in N^{1,p}(\Omega)$ and $\rho \in L^p(\Omega)$ so that $\rho$ is a $p$-weak upper gradient of $u_0$, $\rho_n$ is an upper gradient of $v_n$, $v_n$ converges to $u_0$ in the $L^p(\Omega)$ norm and $\rho_n$ converges to $\rho$ in the $L^p(\Omega)$ norm.

We will show that $J(u_0)=I$, but first we show that $u_0-g \in N_0^{1,p}(\Omega)$.
Notice that $v_n-g \in N_0^{1,p}(\Omega)$ for every $n$. By extending $v_n-g$ as zero in the complement of $\Omega$, we get $v_n-g \in N^{1,p}(X)$. Thus
\begin{equation*}
\| v_n - g \|_{N^{1,p}(X)}
= \| v_n - g \|_{N^{1,p}(\Omega)}
\leq \| g \|_{N^{1,p}(\Omega)} + C,
\end{equation*}
where $C$ is the upper bound of $\| u_n \|_{N^{1,p}(\Omega)}$. Hence the sequence $(v_n-g)_{n=1}^{\infty}$ is bounded in the $N^{1,p}(X)$ norm.
Thus we get from \cite[Lemma 6.2]{bjornbjorn} that there exist $w \in N^{1,p}(X)$ and a sequence of convex combinations of $(v_n-g)_{n=1}^{\infty}$, denoted $(w_n)_{n=1}^{\infty}$, such that $w_n$ converges to $w$ in the $L^p(X)$ norm and $w_n$ converges to $w$ quasieverywhere.
This means that there exists a set $A \subset X$ such that $C_p(A)=0$ and $\lim_{n \to \infty} w_n(x) = w(x)$ for every $x \in X \setminus A$. As $w_n$ is a convex combination of the functions $v_n-g$, we get that $w_n$ is zero in $X \setminus \Omega$, and therefore for every $x \in (X \setminus \Omega) \setminus A$ we have $w(x) = \lim_{n \to \infty} w_n(x) =~0$.
In other words $w=0$ quasieverywhere in $X \setminus \Omega$.
On the other hand $w \in~N^{1,p}(\Omega)$, $u_0-g \in N^{1,p}(\Omega)$ and
\begin{align*}
\|w - (u_0-g) \|_{L^p(\Omega)}
&= \lim_{n \to \infty} \|w_n - u_0 + g \|_{L^p(\Omega)}
\\
&= \lim_{n \to \infty} \left\| \sum_{i = n}^{m_n} \lambda_{i,n} (v_i-u_0) \right\|_{L^p(\Omega)}
\\
&\leq \lim_{n \to \infty} \sum_{i = n}^{m_n} \lambda_{i,n} \| v_i-u_0 \|_{L^p(\Omega)}
= 0,
\end{align*}
because $v_i$ converges to $u_0$ in the $L^p(\Omega)$ norm.
Then $u_0-g=w$ $\mu$-a.e. in $\Omega$ and \cite[Proposition 1.59]{bjornbjorn} implies that $u_0 - g = w$ quasieverywhere in $\Omega$.
Finally by extending $u_0-g$ as zero in the complement of $\Omega$, we get that $u_0-g=w$ quasieverywhere in $X$. Therefore by \cite[Proposition 1.61]{bjornbjorn} we have that $u_0-g \in N^{1,p}(X)$ and therefore $u_0-g \in N_0^{1,p}(\Omega)$.

Now we show that $u_0$ reaches the infimum $I$.
Let $\| g_{u_n} \|_{L^p(\Omega)}$ converge to~$I_1$. Then for any $\eps > 0$ there exists $N$ so that $\| g_{u_n} \|_{L^p(\Omega)} \leq I_1 + \eps$ whenever $n \geq N$. Hence
\begin{equation*}
\| \rho_n \|_{L^p ( \Omega ) }
= \left\| \sum_{i=n}^{m_n} \lambda_{i,n} g_{u_i} \right\|_{L^p( \Omega ) }
\leq \sum_{i=n}^{m_n} \lambda_{i,n} \| g_{u_i} \|_{L^p( \Omega ) }
\leq I_1 + \eps.
\end{equation*}
This means that $\| \rho \|_{L^p ( \Omega )  } = \lim_{n \to \infty} \| \rho_n \|_{L^p ( \Omega )  } \leq I_1$.

Let $\int_{\Omega} f u_n d\mu$ converge to $I_2$.
Then for any $\eps > 0$ there exists $N$ so that $| \int_{\Omega} f u_n d\mu - I_2 | < \eps$ whenever $n \geq N$. Hence
\begin{equation*}
\left| \int_{\Omega} f v_n d\mu - I_2 \right|
= \left| \int_{\Omega} f \sum_{i=n}^{m_n} \lambda_{i,n} u_i d\mu - I_2 \right|
= \left| \sum_{i=n}^{m_n} \lambda_{i,n} \left( \int_{\Omega} f u_i d\mu - I_2 \right) \right|
< \eps
\end{equation*}
and therefore $\int_{\Omega} f v_n d\mu$ converges to $I_2$.
On the other hand $\int_{\Omega} f v_n d\mu$ also converges to $\int_{\Omega} f u_0 d\mu$, because by Hölder's inequality
\begin{equation*}
\left| \int_{\Omega} f v_n d\mu - \int_{\Omega} f u_0 d\mu \right|
\leq \int_{\Omega} |f| |v_n-u_0| d\mu
\leq \| f \|_{L^{p'} ( \Omega )  } \|v_n-u_0\|_{L^p 
( \Omega )  }
\to 0.
\end{equation*}

Thus $\int_{\Omega} f u_0 d\mu = I_2$ and finally
\begin{align*}
I
\leq J(u_0)
&= \int_{\Omega} g_{u_0}^p d\mu - p \int_{\Omega} f u_0 d\mu
\leq \int_{\Omega} \rho^p d\mu - p \int_{\Omega} f u_0 d\mu
\leq I_1^p - p I_2
\\
&= \lim_{n \to \infty} \int_{\Omega} g_{u_n}^p d\mu - p \int_{\Omega} f u_n d\mu
= I.
\end{align*}
This proves that $J$ reaches the infimum at $u_0$.
Therefore $u_0$ is a solution to the problem.
\end{proof}

\section{Stability of solutions}
\label{sect-stability}

Now that we have proved the existence of solutions to the Dirichlet problem, we can ask about their stability.
In this section we prove the stability with respect to the source term and with respect to the boundary data.
We assume that $X$ is geodesic and we fix a differentiable structure of Theorem~\ref{Existence of differentiable structure on mms throu asym gen linear}.
So in this section the gradients are interpreted as Cheeger's gradients.
This is particularly useful, because the Cheeger's gradient is linear and \eqref{mazyalemmacheeger} holds.
In this section solutions are defined as in Definition \ref{solutiondefinitioncheeger}.
To prove the stability, we first need to prove the following lemma.

\begin{lemma}
\label{Uniform bound for the p-energy of the solutions of the p-Laplace equation}
Let $u$ be a solution to the Dirichlet problem for the $p$-Poisson equation (in the Cheeger's setting) with source term $f \in L^{p'}(\Omega)$ and boundary data $g \in N^{1,p}(\Omega)$.
Then
\begin{equation*}
\|\nabla u \|_{L^p(\Omega)}  
\leq 1+ \left( \| \nabla g \|_{L^p(\Omega)} + c_1 \sqrt{N} \|f\|_{L^{p'}(\Omega)} \left( 1 + \| \nabla g \|_{L^p(\Omega)} \right) \right)^{\frac{1}{\min\{1,p-1\}}}.
\end{equation*}
That is, there exists a uniform bound, that is independent of $u$, for the $p$-energy of $u$.
The constants $c_1$ and $N$ are the same as in \eqref{mazyalemmacheeger}.
The right-hand side of the previous inequality will be denoted by $K(f,g)$.
\end{lemma}
\begin{proof}
Clearly the result holds, if $\|\nabla u\|_{L^p(\Omega)} \leq 1$.
Thus let $u$ be a solution of the Dirichlet problem with $ \| \nabla u \|_{L^p( \Omega )} > 1 $.

Notice that $u-g \in N_0^{1,p}(\Omega)$. Thus we can take $\varphi = u - g $ in \eqref{Variational formulation of the p-Laplace equation}, so we have
\begin{equation*}
\| \nabla u \|_{L^p ( \Omega ) }^p
=
\int_{\Omega}|\nabla u|_x^{p-2 } \langle \nabla u , \nabla g \rangle_x + f(u-g) d \mu(x).
\end{equation*}
By applying Hölder's inequality to the above and then \eqref{mazyalemmacheeger}, we obtain
\begin{equation}
\label{estimation of p-energy}
\begin{split}
\| \nabla u \|_{L^p ( \Omega ) }^p
&\leq 
\| \nabla u \|_{L^p ( \Omega ) }^{p-1}
\| \nabla g \|_{L^p ( \Omega ) }
+
\| f \|_{L^{p'} ( \Omega ) }
\| u - g \|_{L^p ( \Omega ) }
\\
&\leq 
\| \nabla u \|_{L^p ( \Omega ) }
^{p-1}
\| \nabla g \|_{L^p ( \Omega ) }
+
c_1 \sqrt{N}
\| f \|_{L^{p'} ( \Omega ) } \| \nabla (u-g) \|_{L^p(\Omega)}
\\
\leq \| \nabla u \|_{L^p(\Omega)}^{p-1} &\| \nabla g \|_{L^p(\Omega)} + c_1 \sqrt{N} \|f\|_{L^{p'}(\Omega)} \left( \| \nabla u \|_{L^p(\Omega)} + \| \nabla g \|_{L^p(\Omega)} \right).
\end{split}
\end{equation}

If $p \geq 2$, we multiply both sides of \eqref{estimation of p-energy} with $\| \nabla u \|_{L^p( \Omega )}^{1-p}$ to obtain
\begin{align*}
\| \nabla u \|_{L^p(\Omega)}
&\leq \| \nabla g \|_{L^p(\Omega)}
\\
&+ c_1 \sqrt{N} \| f \|_{L^{p'}(\Omega)} \left( \| \nabla u \|_{L^p(\Omega)}^{2-p} + \| \nabla g \|_{L^p(\Omega)} \| \nabla u \|_{L^p(\Omega)}^{1-p} \right)
\\
&\leq \| \nabla g \|_{L^p(\Omega)} + c_1 \sqrt{N} \| f \|_{L^{p'}(\Omega)} \left(1 + \| \nabla g \|_{L^p(\Omega)} \right).
\end{align*}

On the other hand if $1 < p < 2$, then by dividing both sides of \eqref{estimation of p-energy} with $\| \nabla u \|_{L^p ( \Omega ) }$, we get
\begin{align*}
\| \nabla u \|_{L^p ( \Omega ) }^{p-1}
&\leq 
\| \nabla u \|_{L^p ( \Omega ) }
^{p-2}
\| \nabla g \|_{L^p ( \Omega ) }
+
c_1 \sqrt{N}
\| f \|_{L^{p'} ( \Omega ) }
\left( 1 + \frac{\| \nabla g \|_{L^p(\Omega)}}{\| \nabla u \|_{L^p(\Omega)}} \right)
\\
&\leq 
\| \nabla g \|_{L^p ( \Omega ) }
+
c_1 \sqrt{N}
\| f \|_{L^{p'} ( \Omega ) }
\left( 1 + \| \nabla g \|_{L^p(\Omega)} \right).
\end{align*}
This completes the proof.
\end{proof}

\begin{theorem}
The Dirichlet problem for the $p$-Poisson equation is stable with respect to perturbations of the source term, that is to say if $u$ is a solution to the Dirichlet problem for the $p$-Poisson equation (in the Cheeger's setting) with source term $f \in L^{p'}(\Omega)$ and boundary data $g \in N^{1,p}(\Omega)$, and $v$ is a solution with source term $h \in L^{p'}(\Omega)$ and boundary data $g$, then
\begin{equation*}
\| u-v \|_{N^{1,p}(\Omega)}
\leq C \|f-h\|_{L^{p'}(\Omega)}^{\frac{1}{p-1}},
\end{equation*}
if $p \geq 2$, and
\begin{equation*}
\|u-v\|_{N^{1,p}(\Omega)}
\leq C \max \left\{ \|f-h\|_{L^{p'}(\Omega)} K(f,g)^{2-p} , \|f-h\|_{L^{p'}(\Omega)}^{\frac{1}{p-1}} \right\},
\end{equation*}
if $1 < p < 2$. Here $K(f,g)$ is the same as in Lemma \ref{Uniform bound for the p-energy of the solutions of the p-Laplace equation} and $C$ is a constant that depends only on $c_1$, $N$ and $p$, where $c_1$ and $N$ are the constants in \eqref{mazyalemmacheeger}.
\end{theorem}

\begin{proof}
Considering \eqref{Variational formulation of the p-Laplace equation} for both solutions $u$ and $v$ with the same $\varphi$ and subtracting them, we have
\begin{equation*}
\int_{\Omega} \langle |\nabla u|_x^{p-2} \nabla u - |\nabla v|_x^{p-2} \nabla v , \nabla \varphi \rangle_x d\mu(x)
= \int_{\Omega} (f-h) \varphi d\mu 
\quad
\forall \quad \varphi \in N_0^{1,p}(\Omega).
\end{equation*}
Since $u$ and $v$ have the same boundary data, we have that $u-g \in N_0^{1,p}(\Omega)$ and $v-g \in N_0^{1,p}(\Omega)$ and therefore $\varphi_0 := u-v = (u-g) - (v-g) \in N_0^{1,p}(\Omega)$. Evaluating the previous equation with $\varphi_0$, we obtain
\begin{equation}
\label{differencewithuandv}
\int_{\Omega} \langle |\nabla u|_x^{p-2} \nabla u - |\nabla v|_x^{p-2} \nabla v , \nabla u - \nabla v \rangle_x d\mu(x)
= \int_{\Omega} (f-h) (u-v) d\mu. 
\end{equation}
From Hölder's inequality and \eqref{mazyalemmacheeger} we have
\begin{equation}
\label{Estimation between the norm of the solutions by the energy of the difference-1}
\begin{split}
\int_{\Omega} (f-h) (u-v) d\mu
&\leq
\| f-h \|_{L^{p'}(\Omega)} 
\| u -  v \|_{L^p(\Omega)}
\\
&\leq c_1 \sqrt{N} \| f-h \|_{L^{p'}(\Omega)} \|\nabla u -\nabla v \|_{L^p(\Omega)}.
\end{split}
\end{equation}
From Theorem \ref{Differentiable structure on mms throu asym gen linear properties} and \eqref{mazyalemmacheeger} we get that
\begin{equation}
\label{newtoniannorm}
\begin{split}
\| u-v \|_{N^{1,p}(\Omega)}
&= \left( \| u-v \|_{L^p(\Omega)}^p + \| g_{u-v} \|_{L^p(\Omega)}^p \right)^{1/p}
\\
&\leq \left( c_1^p \sqrt{N}^p \| \nabla (u-v) \|_{L^p(\Omega)}^p + \sqrt{N}^p \| \nabla (u-v) \|_{L^p(\Omega)}^p \right)^{1/p}
\\
&= \left( c_1^p+1 \right)^{1/p} \sqrt{N} \| \nabla u - \nabla v \|_{L^p(\Omega)}.
\end{split}
\end{equation}

Now assume that $p \geq 2$.
From \eqref{differencewithuandv} and \eqref{Estimation between the norm of the solutions by the energy of the difference-1} together with Lemma \ref{notclarkson} it follows that
\begin{align*}
\int_\Omega
|\nabla u - \nabla v|_x^p d \mu(x)
&\leq \int_{\Omega} 2^{p-1} \langle |\nabla u|_x^{p-2} \nabla u - |\nabla v|_x^{p-2} \nabla v , \nabla u - \nabla v \rangle_x d\mu(x)
\\
&\leq c_1 \sqrt{N} 2^{p-1} \|f-h\|_{L^{p'}(\Omega)} \| \nabla u - \nabla v \|_{L^p(\Omega)}
\\
\implies 
\| \nabla u - &\nabla v \|_{L^p(\Omega)}^{p-1}
\leq c_1 \sqrt{N} 2^{p-1} \|f-h\|_{L^{p'}(\Omega)}.
\end{align*}
From the above estimation and \eqref{newtoniannorm} we get that
\begin{align*}
\| u-v \|_{N^{1,p}(\Omega)}
&\leq \left( c_1^p+1 \right)^{1/p} \sqrt{N} \| \nabla u - \nabla v \|_{L^p(\Omega)}
\\
&\leq 2 \left( c_1^p+1 \right)^{1/p} \sqrt{N}^{\frac{p}{p-1}} \left( c_1 \|f-h\|_{L^{p'}(\Omega)} \right)^{\frac{1}{p-1}}.
\end{align*}
This proves the result for the case $p \geq 2$.

Next we assume that $1<p<2$.
Applying Hölder's inequality with exponents $2/p$ and $2/ ( 2 - p )$, we obtain
\begin{equation}
\label{Estimation for the p-energy of the p-Laplace equation in the case 1<p<2}
\begin{split}
&\int_\Omega |\nabla(u-v)|_x^p d \mu(x)
\\
= &\int_\Omega |\nabla(u-v)|_x^p (|\nabla u|_x+|\nabla v|_x)^{p(p-2)/2} \cdot (|\nabla u|_x+|\nabla v|_x)^{p(2-p) / 2} d \mu(x)
\\
\leq &\left( \int_\Omega |\nabla(u-v)|_x^2(|\nabla u|_x+|\nabla v|_x)^{p-2} d \mu(x) \right)^{p/2} \left\| |\nabla u|_x + |\nabla v|_x \right\|_{L^p(\Omega)}^{p(2-p)/2}
\\
\leq &\left( \int_\Omega \frac{1}{p-1} \langle |\nabla u|_x^{p-2} \nabla u-|\nabla v|_x^{p-2} \nabla v , \nabla u-\nabla v \rangle_x d \mu(x) \right)^{p/2}
\\
&\cdot \left( \| \nabla u \|_{L^p(\Omega)} + \| \nabla v \|_{L^p(\Omega)} \right)^{p(2-p)/2}.
\end{split}
\end{equation}
In the last inequality we used Lemma \ref{notclarkson}.
Using \eqref{differencewithuandv} and \eqref{Estimation between the norm of the solutions by the energy of the difference-1} it follows that
\begin{align*}
\int_\Omega
|\nabla(u-v)|_x^p d \mu(x)
&\leq \left( \frac{1}{p-1} c_1 \sqrt{N} \| f-h \|_{L^{p'}(\Omega)} \|\nabla u -\nabla v \|_{L^p(\Omega)} \right)^{p/2}
\\
&\cdot \left( \| \nabla u \|_{L^p(\Omega)} + \| \nabla v \|_{L^p(\Omega)} \right)^{p(2-p)/2}
\end{align*}
and therefore
\begin{equation}
\label{technicaldetail}
\begin{split}
\left\| \nabla u - \nabla v \right\|_{L^p(\Omega)}
&\leq \frac{c_1 \sqrt{N}}{p-1} \| f - h \|_{L^{p'}(\Omega)} \left( \| \nabla u \|_{L^p(\Omega)} + \| \nabla v \|_{L^p(\Omega)} \right)^{2-p}
\\
\leq \frac{c_1 \sqrt{N}}{p-1} &\|f-h\|_{L^{p'}(\Omega)} \left( 2\left\| \nabla u \right\|_{L^p(\Omega)} + \left\| \nabla v - \nabla u \right\|_{L^p(\Omega)} \right)^{2-p}.
\end{split}
\end{equation}
If $\| \nabla v - \nabla u \|_{L^p(\Omega)} \leq \| \nabla u \|_{L^p(\Omega)}$, then we have
\begin{equation*}
\left\| \nabla u - \nabla v \right\|_{L^p(\Omega)}
\leq \frac{c_1 \sqrt{N}}{p-1} \|f-h\|_{L^{p'}(\Omega)} \left( 3\left\| \nabla u \right\|_{L^p(\Omega)} \right)^{2-p}.
\nonumber
\end{equation*}
By using Lemma \ref{Uniform bound for the p-energy of the solutions of the p-Laplace equation} we obtain
\begin{equation*}
\left\| \nabla u - \nabla v \right\|_{L^p(\Omega)}
\leq \frac{3^{2-p} c_1 \sqrt{N}}{p-1} \|f-h\|_{L^{p'}(\Omega)} K(f,g)^{2-p}.
\end{equation*}
On the other hand if $\| \nabla v - \nabla u \|_{L^p(\Omega)} > \| \nabla u \|_{L^p(\Omega)}$, we can multiply both sides of \eqref{technicaldetail} with $\| \nabla v - \nabla u \|_{L^p(\Omega)}^{p-2}$ to obtain
\begin{align*}
\left\| \nabla u - \nabla v \right\|_{L^p(\Omega)}^{p-1}
&\leq \frac{c_1 \sqrt{N}}{p-1} \|f-h\|_{L^{p'}(\Omega)} \left( \frac{2\left\| \nabla u \right\|_{L^p(\Omega)}}{\left\| \nabla v - \nabla u \right\|_{L^p(\Omega)}} +1 \right)^{2-p}
\\
&\leq \frac{3^{2-p} c_1 \sqrt{N}}{p-1} \|f-h\|_{L^{p'}(\Omega)}.
\end{align*}
In conclusion by using \eqref{newtoniannorm}
\begin{align*}
\|u-v\|_{N^{1,p}(\Omega)}
&\leq (c_1^p+1)^{1/p} \sqrt{N} \| \nabla u - \nabla v\|_{L^p(\Omega)}
\\
&\leq C \max \left\{ \|f-h\|_{L^{p'}(\Omega)} K(f,g)^{2-p} , \|f-h\|_{L^{p'}(\Omega)}^{\frac{1}{p-1}} \right\}.
\end{align*}
This proves the result for the case $1 < p < 2$.
\end{proof}

\begin{theorem}
The Dirichlet problem for the $p$-Poisson equation is stable with respect to perturbations of the boundary data, that is to say if $u$ is a solution to the Dirichlet problem for the $p$-Poisson equation (in the Cheeger's setting) with source term $f \in L^{p'}(\Omega)$ and boundary data $g \in N^{1,p}(\Omega)$, and $v$ is a solution with source term $f$ and boundary data $\hat{g} \in N^{1,p}(\Omega)$, then
\begin{equation*}
\|u-v\|_{N^{1,p}(\Omega)}
\leq C \max \left\{ \| g - \hat{g} \|_{N^{1,p}(\Omega)} , \left\| g - \hat{g} \right\|_{N^{1,p}(\Omega)}^{\frac{1}{\max\{p,2\}}} K(f,g)^{1-\frac{1}{\max\{p,2\}}} \right\}.
\end{equation*}
Here $K(f,g)$ is the same as in Lemma \ref{Uniform bound for the p-energy of the solutions of the p-Laplace equation} and the constant $C$ depends only on $c_1$, $N$ and $p$, where $c_1$ and $N$ are the constants from \eqref{mazyalemmacheeger}.
\end{theorem}

\begin{proof}
Note that $\varphi_1 := u-v+\hat{g}-g \in N_0^{1,p}(\Omega)$, because $u-g \in N_0^{1,p}(\Omega)$ and $v-\hat{g} \in N_0^{1,p}(\Omega)$.
This means that we can apply \eqref{mazyalemmacheeger} to the function~$\varphi_1$.
Thus we get
\begin{equation}
\label{newtoniannorm2}
\begin{split}
&\| u-v \|_{N^{1,p}(\Omega)}
= \left( \|u-v\|_{L^p(\Omega)}^p + \|g_{u-v}\|_{L^p(\Omega)}^p \right)^{1/p}
\\
&\leq \|u-v\|_{L^p(\Omega)} + \|g_{u-v}\|_{L^p(\Omega)}
\\
&\leq \| u-v+\hat{g}-g \|_{L^p(\Omega)} + \| g-\hat{g} \|_{L^p(\Omega)} + \sqrt{N} \|\nabla(u-v)\|_{L^p(\Omega)}
\\
&\leq c_1 \sqrt{N} \| \nabla (u-v+\hat{g}-g) \|_{L^p(\Omega)} + \| g-\hat{g} \|_{L^p(\Omega)} + \sqrt{N} \|\nabla(u-v)\|_{L^p(\Omega)}
\\
&\leq (c_1+1) \sqrt{N} \| \nabla (u-v) \|_{L^p(\Omega)} + c_1 \sqrt{N} \| \nabla (g - \hat{g}) \|_{L^p(\Omega)} + \| g-\hat{g} \|_{L^p(\Omega)}
\\
&\leq (c_1+1) \sqrt{N} \| \nabla u - \nabla v \|_{L^p(\Omega)} + (c_1 \sqrt{N} +1) \| g - \hat{g} \|_{N^{1,p}(\Omega)},
\end{split}
\end{equation}
where we also used Theorem \ref{Differentiable structure on mms throu asym gen linear properties}.
This means that we have the stability, if we can bound $\| \nabla u - \nabla v \|_{L^p(\Omega)}$ with $\| g-\hat{g} \|_{N^{1,p}(\Omega)}$.
By substituting $\varphi_1$ in \eqref{Variational formulation of the p-Laplace equation} for both $u$ and $v$, we obtain
\begin{equation*}
\int_{\Omega} |\nabla u|_x^{p-2} \langle \nabla u , \nabla \varphi_1 \rangle_x d \mu(x)
= \int_{\Omega} f \varphi_1 d \mu
= \int_{\Omega} |\nabla v|_x^{p-2} \langle \nabla v , \nabla \varphi_1 \rangle_x d \mu(x).
\end{equation*}
From Hölder's inequality, we have
\begin{equation}
\label{Stability of the p-Poisson problem when the boundary information vary-2}
\begin{split}
&\int_\Omega \langle |\nabla u|_x^{p-2} \nabla u - |\nabla v|_x^{p-2} \nabla v , \nabla u - \nabla v \rangle_x d \mu(x)
\\
= &\int_{\Omega} \langle |\nabla u|_x^{p-2} \nabla u - |\nabla v|_x^{p-2} \nabla v , \nabla g - \nabla \hat{g} \rangle_x d \mu(x)
\\
\leq &\int_{\Omega} \left| |\nabla u|_x^{p-2} \nabla u - |\nabla v|_x^{p-2} \nabla v \right|_x |\nabla g - \nabla \hat{g}|_x d \mu(x)
\\
\leq &\left\| |\nabla u|_x^{p-2} \nabla u - |\nabla v|_x^{p-2} \nabla v \right\|_{L^{p'}(\Omega)} \left\| \nabla g - \nabla \hat{g} \right\|_{L^p(\Omega)}
\\
\leq &\| \nabla g - \nabla \hat{g} \|_{L^p(\Omega)} \left( \left\| |\nabla u|_x^{p-2} \nabla u \right\|_{L^{p'}(\Omega)} + \left\| |\nabla v|_x^{p-2} \nabla v \right\|_{L^{p'}(\Omega)} \right)
\\
= &\left\| \nabla g - \nabla \hat{g} \right\|_{L^p(\Omega)} \left( \left\| \nabla u \right\|_{L^{p}(\Omega)}^{p-1} + \left\| \nabla v \right\|_{L^{p}(\Omega)}^{p-1} \right).
\end{split}
\end{equation}

Assume now that $p \geq 2$.
From Lemma \ref{notclarkson} and \eqref{Stability of the p-Poisson problem when the boundary information vary-2} we have
\begin{equation}
\label{technicaldetailtwo}
\begin{split}
\int_\Omega &|\nabla u - \nabla v|_x^p d \mu(x)
\leq \int_\Omega 2^{p-1} \langle |\nabla u|_x^{p-2} \nabla u - |\nabla v|_x^{p-2} \nabla v , \nabla u - \nabla v \rangle_x d \mu(x)
\\
&\leq 2^{p-1} \left\| \nabla g - \nabla \hat{g} \right\|_{L^p(\Omega)} \left( \left\| \nabla u \right\|_{L^{p}(\Omega)}^{p-1} + \left\| \nabla v \right\|_{L^{p}(\Omega)}^{p-1} \right)
\\
&\leq 2^{p-1} \left\| \nabla g - \nabla \hat{g} \right\|_{L^p(\Omega)} \left( \left\| \nabla u \right\|_{L^p(\Omega)} + \left\| \nabla v \right\|_{L^p(\Omega)} \right)^{p-1}
\\
&\leq 2^{p-1} \left\| \nabla g - \nabla \hat{g} \right\|_{L^p(\Omega)} \left( 2\left\| \nabla u \right\|_{L^p(\Omega)} + \left\| \nabla u - \nabla v \right\|_{L^p(\Omega)} \right)^{p-1}.
\end{split}
\end{equation}
If $\| \nabla v - \nabla u \|_{L^p(\Omega)} \leq \| \nabla u \|_{L^p(\Omega)}$, then we have
\begin{equation*}
\left\| \nabla u - \nabla v \right\|_{L^p(\Omega)}^p
\leq 6^{p-1} \left\| \nabla g - \nabla \hat{g} \right\|_{L^p(\Omega)} \left\| \nabla u \right\|_{L^p(\Omega)}^{p-1}.
\end{equation*}
By using Lemma \ref{Uniform bound for the p-energy of the solutions of the p-Laplace equation} we obtain
\begin{equation*}
\left\| \nabla u - \nabla v \right\|_{L^p(\Omega)}
\leq 6 \left\| \nabla g - \nabla \hat{g} \right\|_{L^p(\Omega)}^{\frac{1}{p}} K(f,g)^{1-\frac{1}{p}}.
\end{equation*}
On the other hand if $\| \nabla v - \nabla u \|_{L^p(\Omega)} > \| \nabla u \|_{L^p(\Omega)}$, we can multiply both sides of \eqref{technicaldetailtwo} with $\| \nabla v - \nabla u \|_{L^p(\Omega)}^{1-p}$ to obtain
\begin{align*}
\left\| \nabla u - \nabla v \right\|_{L^p(\Omega)}
&\leq 2^{p-1} \left\| \nabla g - \nabla \hat{g} \right\|_{L^p(\Omega)} \left( \frac{2 \| \nabla u \|_{L^p(\Omega)}}{\| \nabla v - \nabla u \|_{L^p(\Omega)}} +1 \right)^{p-1}
\\
&\leq 6^{p-1} \left\| \nabla g - \nabla \hat{g} \right\|_{L^p(\Omega)}.
\end{align*}
In conclusion by using Theorem \ref{Differentiable structure on mms throu asym gen linear properties} and \eqref{newtoniannorm2} we get
\begin{align*}
\|u-v\|_{N^{1,p}(\Omega)}
&\leq (c_1+1) \sqrt{N} \| \nabla u - \nabla v \|_{L^p(\Omega)} + (c_1 \sqrt{N} +1) \| g - \hat{g} \|_{N^{1,p}(\Omega)}
\\
&\leq C \max \left\{ \left\| \nabla g - \nabla \hat{g} \right\|_{L^p(\Omega)}^{\frac{1}{p}} K(f,g)^{1-\frac{1}{p}} , \left\| \nabla g - \nabla \hat{g} \right\|_{L^p(\Omega)} \right\}
\\
&+ C \| g - \hat{g} \|_{N^{1,p}(\Omega)}
\\
&\leq C \max \left\{ \| g - \hat{g} \|_{N^{1,p}(\Omega)} , \left\| g - \hat{g} \right\|_{N^{1,p}(\Omega)}^{\frac{1}{p}} K(f,g)^{1-\frac{1}{p}} \right\},
\end{align*}
where the constant $C$ depends only on $c_1$, $N$ and $p$.
This proves the result for the case $p \geq 2$.

Now assume that $1<p<2$.
Notice that \eqref{Estimation for the p-energy of the p-Laplace equation in the case 1<p<2} holds.
Thus by using \eqref{Stability of the p-Poisson problem when the boundary information vary-2}, we obtain
\begin{equation}
\label{technicaldetailthree}
\begin{split}
&\int_\Omega |\nabla(u-v)|_x^p d \mu(x)
\\
&\leq \left( \int_\Omega \frac{1}{p-1} \langle |\nabla u|_x^{p-2} \nabla u-|\nabla v|_x^{p-2} \nabla v , \nabla u-\nabla v \rangle_x d \mu(x) \right)^{p/2}
\\
&\cdot \left( \| \nabla u \|_{L^p(\Omega)} + \| \nabla v \|_{L^p(\Omega)} \right)^{p(2-p)/2}
\\
&\leq C \|\nabla g  - \nabla \hat{g} \|_{L^p(\Omega)}^{p/2} \left( \| \nabla u \|_{L^p(\Omega)}^{p-1} + \| \nabla v \|_{L^p(\Omega)}^{p-1} \right)^{p/2}
\\
&\cdot \left( \| \nabla u \|_{L^p(\Omega)} + \| \nabla v \|_{L^p(\Omega)} \right)^{p(2-p)/2}
\\
&\leq C \|\nabla g  - \nabla \hat{g} \|_{L^p(\Omega)}^{p/2} \left( \| \nabla u \|_{L^p(\Omega)} + \| \nabla v \|_{L^p(\Omega)} \right)^{p/2}
\\
&\leq C \|\nabla g  - \nabla \hat{g} \|_{L^p(\Omega)}^{p/2} \left( 2 \| \nabla u \|_{L^p(\Omega)} + \| \nabla u - \nabla v \|_{L^p(\Omega)} \right)^{p/2},
\end{split}
\end{equation}
where the constant $C$ depends only on $p$.
If $\| \nabla v - \nabla u \|_{L^p(\Omega)} \leq \| \nabla u \|_{L^p(\Omega)}$, then we obtain by using Lemma \ref{Uniform bound for the p-energy of the solutions of the p-Laplace equation}
\begin{align*}
\left\| \nabla u - \nabla v \right\|_{L^p(\Omega)}^p
&\leq C \left\| \nabla g - \nabla \hat{g} \right\|_{L^p(\Omega)}^{p/2} \left\| \nabla u \right\|_{L^p(\Omega)}^{p/2}
\\
&\leq C \left\| \nabla g - \nabla \hat{g} \right\|_{L^p(\Omega)}^{p/2} K(f,g)^{p/2}
\\
\implies \left\| \nabla u - \nabla v \right\|_{L^p(\Omega)}
&\leq C \left\| \nabla g - \nabla \hat{g} \right\|_{L^p(\Omega)}^{1/2} K(f,g)^{1/2}.
\end{align*}
On the other hand if $\| \nabla v - \nabla u \|_{L^p(\Omega)} > \| \nabla u \|_{L^p(\Omega)}$, we can multiply both sides of \eqref{technicaldetailthree} with $\| \nabla v - \nabla u \|_{L^p(\Omega)}^{-p/2}$ to obtain
\begin{align*}
\left\| \nabla u - \nabla v \right\|_{L^p(\Omega)}^{p/2}
&\leq C \left\| \nabla g - \nabla \hat{g} \right\|_{L^p(\Omega)}^{p/2} \left( \frac{2 \| \nabla u \|_{L^p(\Omega)}}{ \| \nabla v - \nabla u \|_{L^p(\Omega)}} +1 \right)^{p/2}
\\
&\leq C \left\| \nabla g - \nabla \hat{g} \right\|_{L^p(\Omega)}^{p/2}
\\
\implies \| \nabla u - \nabla v \|_{L^p(\Omega)}
&\leq C \left\| \nabla g - \nabla \hat{g} \right\|_{L^p(\Omega)}.
\end{align*}
In conclusion by using Theorem \ref{Differentiable structure on mms throu asym gen linear properties} and \eqref{newtoniannorm2} we get
\begin{align*}
&\|u-v\|_{N^{1,p}(\Omega)}
\leq (c_1+1) \sqrt{N} \| \nabla u - \nabla v \|_{L^p(\Omega)} + (c_1 \sqrt{N} +1) \| g - \hat{g} \|_{N^{1,p}(\Omega)}
\\
&\leq C \max \left\{ \left\| \nabla g - \nabla \hat{g} \right\|_{L^p(\Omega)}^{1/2} K(f,g)^{1/2} , \left\| \nabla g - \nabla \hat{g} \right\|_{L^p(\Omega)} \right\} + C \| g-\hat{g} \|_{N^{1,p}(\Omega)}
\\
&\leq C \max \left\{ \| g - \hat{g} \|_{N^{1,p}(\Omega)} , \left\| g - \hat{g} \right\|_{N^{1,p}(\Omega)}^{1/2} K(f,g)^{1/2} \right\},
\end{align*}
where the constant $C$ depends only on $c_1$, $N$ and $p$.
This proves the result for the case $1 < p < 2$.
\end{proof}

Now we are ready to prove the general stability result.

\begin{theorem}
Let $(f_k)_ {k=1}^{\infty}$ be a sequence of $L^{p'}(\Omega)$ functions that converges to $f \in L^{p'}(\Omega)$ in the $L^{p'}(\Omega)$ norm.
Similarly let $(g_k)_{k=1}^{\infty}$ be a sequence of $N^{1,p}(\Omega)$ functions that converges to $g \in N^{1,p}(\Omega)$ in the $N^{1,p}(\Omega)$ norm.
For every $k \geq 1$ let $u_k$ be a solution to the Dirichlet problem for the $p$-Poisson equation (in the Cheeger's setting) with source term $f_k$ and boundary data~$g_k$, and let $u$ be a solution with source term $f$ and boundary data $g$. Then $(u_k)_{k=1}^{\infty}$ converges to $u$ in the $N^{1,p}(\Omega)$ norm.
\end{theorem}

\begin{proof}
For every $k \geq 1$ let $v_k$ be a solution to the Dirichlet problem for the $p$-Poisson equation (in the Cheeger's setting) with source term $f_k$ and boundary data $g$.
Notice that the solution $v_k$ shares the source term with~$u_k$, while $v_k$ shares the boundary data with $u$.
Thus we can use the stability results when one of the parameters is fixed - the source term or the boundary data. 
Then by using the triangle inequality we get
\begin{align*}
\| u-u_k \|_{N^{1,p}(\Omega)}
&\leq \| u-v_k \|_{N^{1,p}(\Omega)} + \| v_k-u_k \|_{N^{1,p}(\Omega)}
\\
&\leq C \|f-f_k\|_{L^{p'}(\Omega)}^{\frac{1}{p-1}}
\\
&+ C \max \left\{ \| g - g_k \|_{N^{1,p}(\Omega)} , \left\| g - g_k \right\|_{N^{1,p}(\Omega)}^{\frac{1}{p}} K(f_k,g)^{1-\frac{1}{p}} \right\},
\end{align*}
if $p \geq 2$.
If $1 < p < 2$, we get
\begin{align*}
\| u-u_k \|_{N^{1,p}(\Omega)}
&\leq \| u-v_k \|_{N^{1,p}(\Omega)} + \| v_k-u_k \|_{N^{1,p}(\Omega)}
\\
&\leq C \max \left\{ \|f-f_k\|_{L^{p'}(\Omega)} K(f,g)^{2-p} , \|f-f_k\|_{L^{p'}(\Omega)}^{\frac{1}{p-1}} \right\}
\\
&+ C \max \left\{ \| g - g_k \|_{N^{1,p}(\Omega)} , \left\| g - g_k \right\|_{N^{1,p}(\Omega)}^{\frac{1}{2}} K(f_k,g)^{\frac{1}{2}} \right\}.
\end{align*}
In both cases we get the convergence, because $(K(f_k,g))_{k=1}^{\infty}$ also converges to $K(f,g)$, when $f_k \to f$ in the $L^{p'}$ norm. This completes the proof.
\end{proof}

An immediate corollary of the stability results is the following.

\begin{corollary}
The solution to the Dirichlet problem for the $p$-Poisson equation (in the Cheeger's setting) is unique.
\end{corollary}


\begin{thebibliography}{100}


\bibitem{Benatti2025}
L.\ Benatti and I.\ Y.\ Violo,
\emph{Second-order estimates for the $p$-Laplacian in RCD spaces},
J. Differential Equations \textbf{439} (2025).


\bibitem{bjornbjorn}
A. Björn and J. Björn,
\emph{Nonlinear potential theory on metric spaces},
EMS Tracts Math., volume 17, European Mathematical Society (EMS), Zürich, 2011.


\bibitem{Bungert2023}
L. Bungert,
\emph{The inhomogeneous $p$-Laplacian equation with Neumann boundary conditions in the limit $p\to \infty$},
Adv. Contin. Discrete Models (2023), Paper No. 8.


\bibitem{Capogna2026}
L. Capogna, R. Gibara, R. Korte and N. Shanmugalingam,
\emph{Fractional $p$-Laplacians via Neumann problems in unbounded metric measure spaces},
Preprint, arXiv:2410.18883, 2024.


\bibitem{Capogna2022}
L. Capogna, J. Kline, R. Korte, N. Shanmugalingam and M. Snipes,
\emph{Neumann problems for $p$-harmonic functions, and induced nonlocal operators in metric measure spaces},
Amer. J. Math. \textbf{147} (2025), no. 6, 1653–1711.


\bibitem{Cheeger1999}
J.\ Cheeger,
\emph{Differentiability of Lipschitz functions on metric measure spaces},
Geom. Funct. Anal. \textbf{9} (1999), no. 3, 428--517.


\bibitem{Cianchi2018}
A.\ Cianchi and V.\ G.\ Maz'ya,
\emph{Second-order two-sided estimates in nonlinear elliptic problems},
Arch. Ration. Mech. Anal. \textbf{229} (2018), no. 2, 569--599.




\bibitem{Folland1999}
G.\ B.\ Folland,
\emph{Real analysis: Modern techniques and their applications},
Pure Appl. Math.
Wiley–Intersci. Publ. John Wiley \& Sons, Inc., New York, 1999, 2nd ed.


\bibitem{Garain2020}
P.\ Garain,
\emph{Properties of solutions to some weighted $p$-Laplacian equation},
Opuscula Math. \textbf{40} (2020), no. 4, 483--494.


\bibitem{Gong2012}
J.\ Gong, J.\ J.\ Manfredi and M.\ Parviainen,
\emph{Nonhomogeneous variational problems and quasi-minimizers on metric spaces},
Manuscripta Math. \textbf{137} (2012), no. 1--2, 247--271.


\bibitem{Heinonen1993}
J. Heinonen, T. Kilpeläinen and O. Martio,
\emph{Nonlinear potential theory of degenerate elliptic equations},
Oxford Math. Monogr.
Oxford Sci. Publ.
The Clarendon Press, Oxford University Press, New York, 1993.


\bibitem{Heinonen2015}
J. Heinonen, P. Koskela, N. Shanmugalingam and J. T. Tyson,
\emph{Sobolev spaces on metric measure spaces},
Cambridge University Press, Cambridge, 2015.






\bibitem{Kolokoltsov2019}
V. Kolokoltsov,
\emph{Differential equations on measures and functional spaces},
Birkhäuser Adv. Texts Basler Lehrbücher, Birkhäuser/Springer, Cham, 2019.


\bibitem{Lindgren2015}
E.\ Lindgren and P.\ Lindqvist,
\emph{Regularity of the $p$-Poisson equation in the plane},
J. Anal. Math. \textbf{132} (2017), 217–228.


\bibitem{Lindqvist2017}
P. Lindqvist,
\emph{Notes on the $p$-Laplace equation (second edition)},
Rep. Univ. Jyväskylä Dep. Math. Stat., 161, University of Jyväskylä, Jyväskylä, 2017.




\bibitem{Mikkonen1996}
P. Mikkonen,
\emph{On the Wolff potential and quasilinear elliptic equations involving measures},
Ann. Acad. Sci. Fenn. Math. Diss. No. 104 (1996).


\bibitem{nages}
N. Shanmugalingam,
\emph{Harmonic functions on metric spaces},
Illinois J. Math. {\bf 45} (2001), no. 3, 1021-1050.
\end{thebibliography}
\end{document}